\RequirePackage{fix-cm}

\documentclass[a4paper, oneside]{amsart}

\usepackage[stretch=10, shrink=10]{microtype}
\usepackage{amsmath, amssymb, amsthm}
\usepackage{graphicx}
\usepackage[margin=1in]{geometry}
\usepackage{float}
\usepackage{cases}
\usepackage{enumitem}
\usepackage{tikz}
\usepackage{fixltx2e}
\usepackage[utf8]{inputenc}
\usepackage{makecell}
\usepackage{booktabs}
\usepackage[square, numbers]{natbib}
\usepackage[colorlinks]{hyperref}
\usepackage{doi}

\newcommand*\diff{\mathop{}\!\mathrm{d}}

\newcommand{\circlenum}[1]{\raisebox{.5pt}{\textcircled{\raisebox{-.9pt} {#1}}}}
\newcommand{\re}{\mathrm{Re}\,}
\newcommand{\im}{\mathrm{Im}\,}
\newcommand{\circlepm}{ 
  \mathbin{
    \mathchoice
      {\buildcirclepm{\displaystyle}}
      {\buildcirclepm{\textstyle}}
      {\buildcirclepm{\scriptstyle}}
      {\buildcirclepm{\scriptscriptstyle}}
  }
}
\newcommand\buildcirclepm[1]{%
  \begin{tikzpicture}[baseline=(X.base), inner sep=0, outer sep=0]
    \node[draw,circle] (X)  {$#1\pm$};
  \end{tikzpicture}%
}
\newcommand{\resp}[1]{\textup{(}resp. #1\textup{)}}

\numberwithin{equation}{section}

\theoremstyle{plain}
\newtheorem{lemm}{Lemma}[section]
\newtheorem*{coro}{Corollary}
\newtheorem{coro2}{Corollary}[section]
\newtheorem{prop}{Proposition}[section]
\newtheorem{fact}{Fact}
\newtheorem{theorem}{Theorem}

\theoremstyle{definition}

\theoremstyle{remark}
\newtheorem{remark}{Remark}

\begin{document}

\title{Deformation and singularities of maximal surfaces with planar curvature lines}

\author{Joseph Cho}
\address{\newline Department of mathematics, Faculty of science\newline Kobe university\newline Rokkodai-cho 1-1, Nada-ku, Kobe-shi, Hyogo-ken, 657-8501\newline Japan}
\email{joseph.cho.kobe-u.ac.jp}

\author{Yuta Ogata}
\address{\newline Department of Science and Technology\newline National Institute of Technology, Okinawa College \newline Henoko 905, Nago-shi \newline Okinawa-ken, 905-2171 Japan}
\email{y.ogata@okinawa-ct.ac.jp}


\maketitle

\begin{abstract}
Minimal surfaces with planar curvature lines in the Euclidean space have been studied since the late 19th century. On the other hand, the classification of maximal surfaces with planar curvature lines in the Lorentz-Minkowski space has only recently been given. In this paper, we use an alternative method not only to refine the classification of maximal surfaces with planar curvature lines, but also to show that there exists a deformation consisting exactly of all such surfaces. Furthermore, we investigate the types of singularities that occur on maximal surfaces with planar curvature lines. Finally, by considering the conjugate of maximal surfaces with planar curvature lines, we obtain analogous results for maximal surfaces that are also affine minimal surfaces.
\end{abstract}


\section{Introduction}
Minimal surfaces with planar curvature lines in the Euclidean space $\mathbb{R}^3$ have been studied extensively since the late 19th century by the likes of Bonnet, Enneper, and Eisenhart \cite{bonnet_observations_1855, enneper_untersuchungen_1878, eisenhart_treatise_1909}. In his book, Nitsche summarized the result, that minimal surfaces with planar curvature lines must be a piece of one, and only one, of either the plane, Enneper surface, catenoid, or a surface in the Bonnet family \cite{nitsche_lectures_1989}. Furthermore, many works have been published concerning minimal surfaces with spherical curvature lines, treating planes as spheres with infinite radius \cite{dobriner_minimalflachen_1887, wente_constant_1992}.

On the other hand, the classification of maximal surfaces with planar curvature lines in Lorentz-Minkowski space $\mathbb{R}^{2,1}$ has not been given until recently \cite{leite_surfaces_2015}. Leite developed an approach using orthogonal systems of cycles on the hyperbolic 2-space $\mathbb{H}^2$, and used the fact that families of planar curvature lines transform into orthogonal families of cycles on $\mathbb{H}^2$ under its analogue of the Gauss map. Then she obtained the data for the following Weierstrass-type representation for maximal surfaces as first stated in \cite{kobayashi_maximal_1983}, later refined to include singularities in \cite{umehara_maximal_2006}.

\begin{fact}[Weierstrass-type representation theorem for maximal surfaces]
Any conformal maximal surface $X: \Sigma \subset \mathbb{C} \to \mathbb{R}^{2,1}$ can be locally represented as
\[X = \re\int(1+h^2, i(1-h^2), -2h)\eta\,\diff z\]
over a simply-connected domain $\Sigma$ on which $h$ is meromorphic, while $\eta$ and $h^2\eta$ are holomorphic.
\end{fact}
\noindent (We note here that for a conformal maximal surface with Weierstrass data $(h, \eta \diff z)$, one obtains the associated family of maximal surfaces via Weierstrass data $(h, \lambda^{-2} \eta \diff z)$ for $\lambda \in \mathbb{S}^1 = \{ \lambda \in \mathbb{C} : |\lambda|^2 = 1\}$.) Using the above representation, Leite produced the following classification and their respective Weierstrass data $(h, \eta \diff z)$.
\begin{fact}[\cite{leite_surfaces_2015}]
A maximal surface in Lorentz-Minkowski space $\mathbb{R}^{2,1}$ with planar curvature lines must be a piece of one, and only one of
	\begin{itemize}
		\item[$\bullet$] plane, with Weierstrass data $(0, 1 \diff z)$,
		\item[$\bullet$] Enneper surface of first kind, with Weierstrass data $(z, 1 \diff z)$,
		\item[$\bullet$] Enneper surface of second kind, with Weierstrass data $\left(\frac{1-z}{1+z}, -\frac{(1+z)^2}{2} \diff z\right)$, or one member of its associated family,
		\item[$\bullet$] catenoid of first kind, with Weierstrass data $(e^z, e^{-z} \diff z)$,
		\item[$\bullet$] catenoid of second kind, with Weierstrass data $\left(\frac{1-e^z}{1+e^z}, -1-\cosh{z} \diff z \right)$, or
		\item[$\bullet$] one surface in the Bonnet family, with Weierstrass data $\{(e^z + t, e^{-z} \diff z), t > 0\}$
	\end{itemize}
up to isometries and homotheties of $\mathbb{R}^{2,1}$.
\end{fact}

To study maximal surfaces with planar curvature lines, we start by proposing an alternative method to using Leite's method. In Section \ref{sec:abresch}, we closely follow the method used in \cite{cho_deformation_2017}, which was modeled after techniques used in \cite{abresch_constant_1987}, \cite{barthel_thomsensche_1980}, \cite{walter_explicit_1987}, and \cite{wente_constant_1992}. First, we obtain and solve a system of partial differential equations for the metric function using the zero mean curvature condition and the planar curvature line condition. Then, from the metric function, we find the normal vector to the surface by using the notion of axial directions. From the normal vector, we recover the Weierstrass data and the parametrizations of maximal surfaces with planar curvature lines, allowing us to obtain a complete classification (see Theorem \ref{theo:classification}).

In fact, the axial directions play a crucial role in this paper, as they allow us not only to further classify maximal Bonnet-type surfaces into three types, but also to attain deformations consisting of the surfaces under consideration. In Section \ref{sec:deformation}, we investigate those deformations, and show that there exists a single continuous deformation consisting exactly of the maximal surfaces with planar curvature lines (see Theorem \ref{theo:oneDeformation}, Fig. \ref{fig:deformationDiag}, and Fig. \ref{fig:surfTheta}).

On the other hand, the notion of maxfaces as maximal surfaces in $\mathbb{R}^{2,1}$ with singularities was introduced in \cite{umehara_maximal_2006}, and various types of singularities appearing on maximal surfaces have been studied in various works \cite{kobayashi_maximal_1984, fernandez_space_2005, umehara_maximal_2006, kim_prescribing_2007, fujimori_singularities_2008, ogata_duality_}. Since the singularities of maximal catenoids and the maximal Enneper-type surface were investigated in \cite{kobayashi_maximal_1984, umehara_maximal_2006, fujimori_singularities_2008}, in Section \ref{sec:singularities}, we recognize the types of singularities for maximal Bonnet-type surfaces using the criteria introduced in \cite{umehara_maximal_2006}, \cite{fujimori_singularities_2008}, and \cite{ogata_duality_}, and specify the types of singularities appearing in maximal surfaces with planar curvature lines (see Theorem \ref{theo:singularityType}, Fig. \ref{fig:singDiag1}, and Fig. \ref{fig:singDiag2}).

Finally in Section \ref{sec:thomsen}, we apply the results in Section \ref{sec:deformation} and Section \ref{sec:singularities} to maximal surfaces that are also affine minimal surfaces. Thomsen studied minimal surfaces in $\mathbb{R}^3$ that are also affine minimal surfaces, and mentioned that such surfaces are conjugate surfaces of minimal surfaces with planar curvature lines \cite{thomsen_uber_1923}. Manhart has shown  that the analogous result holds true for the maximal case in $\mathbb{R}^{2,1}$ \cite{manhart_bonnet-thomsen_2015}, and we use that result to consider the deformations and singularities of maximal surfaces that are also affine minimal surfaces (see Corollary \ref{coro:thomsen1}, Corollary \ref{coro:thomsen2}, Fig. \ref{fig:surfConj}, and Fig. \ref{fig:singDiagConj1}).


\section{Classification of maximal surfaces with planar curvature lines}\label{sec:abresch}
In this section, we would like to obtain a complete classification of maximal surfaces with planar curvature lines by using the Weierstrass-type representation. We use an alternative method to orthogonal systems of cycles to recover the Weierstrass data as follows: First, from the zero mean curvature condition and planar curvature line condition, we obtain and solve a system of partial differential equations for the metric function. Then using the explicit solutions for the metric function, we recover the Weierstrass data and the parametrization by calculating the unit normal vector.
\subsection{Maximal surface theory}
Let $\mathbb{R}^{2,1}$ be Lorentz-Minkowski space with Lorentzian metric
	\[
		\langle (x_1,x_2,x_0), (y_1,y_2,y_0)\rangle:=x_1y_1+x_2y_2-x_0y_0.
	\]
In addition, let $\Sigma$ be a simply-connected domain with coordinates $(u,v)\in\Sigma\subset\mathbb{R}^2$. Throughout the paper, we identify $\mathbb{R}^2$ with the set of complex numbers $\mathbb{C}$ via $(u,v) \leftrightarrow z := u + i v$ where $i = \sqrt{-1}$. Let $X:\Sigma \to \mathbb{R}^{2,1}$ be a conformally immersed spacelike surface. Since $X(u,v)$ is conformal, the induced metric $\diff s^2$ is represented as
	\[
		\diff s^2=\rho^2\,(\diff u^2+\diff v^2)
	\]
for some function $\rho:\Sigma \to \mathbb{R}_+$, where $\mathbb{R}_+$ is the set of positive real numbers.

We choose the timelike unit normal vector field $N:\Sigma \to \mathbb{H}^2$ of $X$, where $\mathbb{H}^2$ is the two-sheeted hyperboloid in $\mathbb{R}^{2,1}$ (cf. \cite[(1.2)]{umehara_maximal_2006}), i.e.\
	\[
		\mathbb{H}^2 = \mathbb{H}^2_+ \cup \mathbb{H}^2_-
	\]
for
	\[
		 \mathbb{H}^2_+ := \{ x \in \mathbb{R}^{2,1} : \langle x, x \rangle = -1, x_0 > 0\} \quad\text{and}\quad \mathbb{H}^2_- := \{ x \in \mathbb{R}^{2,1} : \langle x, x \rangle = -1, x_0 < 0\}.
	\]
Now, let $X(u,v)$ be a non-planar umbilic-free maximal surface on the domain $\Sigma$. By \cite[Lemma 2.3.2]{bobenko_painleve_2000} (see also \cite{bobenko_integrable_1990, bobenko_surfaces_1991}), we may then further assume that $(u,v)$ are conformal curvature line (or isothermic) coordinates, and that the Hopf differential factor
	\[
		Q := \langle X_{zz}, N \rangle = -\frac{1}{2}
	\]
without loss of generality. Hence, the Gauss-Weingarten equations for the maximal case are the following:
\begin{equation}\label{eqn:gaussW}
	\begin{cases}
       	 	X_{uu} = \frac{\rho_u}{\rho} X_u -\frac{\rho_v}{\rho} X_v +N,\quad X_{vv} = -\frac{\rho_u}{\rho} X_u + \frac{\rho_v}{\rho} X_v -N\\
		X_{uv} = \frac{\rho_v}{\rho} X_u + \frac{\rho_u}{\rho} X_v, \quad N_u = \frac{1}{\rho^2} X_u, \quad N_v = -\frac{1}{\rho^2} X_v
	\end{cases},
\end{equation}
while the integrability condition, or the Gauss equation, becomes
	\[
		\rho\cdot\Delta\rho - (\rho_u^2 + \rho_v^2) + 1 = 0
	\]
where $\Delta = \partial_u^2 + \partial_v^2$.
Finally, changing $Q\mapsto\lambda^{-2}Q$ for $\lambda\in\mathbb{S}^1\subset\mathbb{C}$, we obtain the associated family of $X(u,v)$. In particular, if $\lambda^{-2} = \pm i$ then the new surface is called the \emph{conjugate} of the original surface.

\subsection{The planar curvature condition and analytic classification}\label{subsec:max_Abresch}
Now, we impose the planar curvature line condition on a maximal surface. First, we consider the relationship between the planar curvature line condition and the metric function.

\begin{lemm}
For a non-planar umbilic-free maximal surface $X(u,v)$, the following statements are equivalent:
	\begin{enumerate}
		\item $u$-curvature lines are planar.
		\item $v$-curvature lines are planar.
		\item $\rho_{uv}= 0$.
	\end{enumerate}
\end{lemm}

\begin{proof}
Similar to the proof of Lemma 2.1 of \cite{cho_deformation_2017}, we can show this by calculating the conditions $\det(X_u, X_{uu}, X_{uuu}) = 0$ and $\det(X_v, X_{vv}, X_{vvv}) = 0$ using \eqref{eqn:gaussW}. 
\end{proof}

Therefore, finding all non-planar umbilic-free maximal surfaces with planar curvature lines is equivalent to finding solutions to the following system of partial differential equations:
	\begin{subnumcases}{\label{eqn:max_pde}}
		\rho\cdot\Delta\rho-(\rho_u^2+\rho_v^2)+1=0 \label{eqn:max_pde1} &\text{(Gauss equation for maximal surfaces),}\\
		\rho_{uv}= 0 \label{eqn:max_pde2} &\text{(planar curvature line condition)}.
	\end{subnumcases}
To solve the above system, we note that \eqref{eqn:max_pde1} and \eqref{eqn:max_pde2} can be reduced to a system of ordinary differential equations as follows.

\begin{lemm}\label{lemm:max_solutionRho}
For a solution $\rho: \Sigma \to \mathbb{R}_+$ to \eqref{eqn:max_pde1} and \eqref{eqn:max_pde2}, there exist real-valued functions $f(u)$ and $g(v)$ such that
	\begin{subnumcases}{\label{eqn:max_recover}}
		\rho_u =f(u) \label{eqn:max_recover1}, \\
		\rho_v =g(v) \label{eqn:max_recover2}.
	\end{subnumcases}
Furthermore, $\rho(u,v)$ can be explicitly written in terms of $f(u)$ and $g(v)$ as follows:
\begin{description}
    \item[Case (1):] If $\Delta\rho$ is nowhere zero on $\Sigma$,
    \begin{equation} \label{eqn:max_solutionRho1}
    	\rho(u,v)= \frac{f(u)^2 + g(v)^2-1}{f_u(u) + g_v(v)},
    \end{equation}
    where $f(u)$ and $g(v)$ satisfy the following system of ordinary differential equations:
    	\begin{subnumcases}{\label{eqn:max_ode}}
    		(f_u(u))^2 = (d-c) f(u)^2 + c \label{eqn:max_ode1}\\
    		f_{uu}(u)= (d- c) f(u) \label{eqn:max_ode2}\\
    		(g_v(v))^2 = (c- d) g(v)^2 + d \label{eqn:max_ode3}\\
    		g_{vv}(v)= (c- d) g(v) \label{eqn:max_ode4}
    	\end{subnumcases}
    for real constants $c$ and $d$ such that $c^2 + d^2 \neq 0$.

    \item[Case (2):] If $\Delta\rho \equiv 0$ on $\Sigma$, i.e.\ $\Delta\rho$ is identically zero on $\Sigma$,
    \begin{equation} \label{eqn:max_solutionRho2}
    	\rho(u,v) = (\cos \phi) \cdot u + (\sin \phi) \cdot v.
    \end{equation}
    where $f(u) = \sin \phi$ and $g(v) = \cos \phi$ for some constant $\phi \in [0, 2\pi)$.

\end{description}
\end{lemm}

\begin{proof}
%
Arguments to prove \eqref{eqn:max_recover1}, \eqref{eqn:max_recover2}, \eqref{eqn:max_solutionRho1}, and \eqref{eqn:max_ode1}--\eqref{eqn:max_ode4} are similar to those in the proof of Lemma 2.2 in \cite{cho_deformation_2017}; here we only give a short outline of these arguments.

Integrating \eqref{eqn:max_pde2} once with respect to $u$ and once with respect to $v$ gives \eqref{eqn:max_recover1} and \eqref{eqn:max_recover2}. Now first assume that $\Delta \rho$ is not identically equal to zero. Then we can choose a point $(u_0, v_0)$ such that $\rho(u_0, v_0) \neq 0$, implying that we can choose a neighborhood $\Sigma \subset \mathbb{R}^2$ of $(u_0, v_0)$ such that $\Delta\rho$ is nowhere zero on $\Sigma$. On such $\Sigma$, we can use \eqref{eqn:max_pde1} with \eqref{eqn:max_recover1} and \eqref{eqn:max_recover2} to obtain \eqref{eqn:max_solutionRho1}. Finally, using \eqref{eqn:max_recover1} and \eqref{eqn:max_recover2} with \eqref{eqn:max_solutionRho1} and integrating with respect to $u$ and $v$, respectively, we obtain \eqref{eqn:max_ode1}--\eqref{eqn:max_ode4}.

Now assume that $\Delta \rho$ is identically equal to zero on some simply-connected domain $\Sigma \subset \mathbb{R}^2$. Since $\Delta \rho = f_u(u)+g_v(v)\equiv 0$,
$f(u)^2 + g(v)^2=1$
for all $u$ and $v$. This implies that both $f(u)$ and $g(v)$ are constant, and we can set $f:=\cos\phi$ and $g:=\sin\phi$ for some constant $\phi \in [0, 2\pi)$. Solving \eqref{eqn:max_recover1} and \eqref{eqn:max_recover2}, we obtain \eqref{eqn:max_solutionRho2}.
\end{proof}

We would now like to solve for $f(u)$ and $g(v)$ satisfying \eqref{eqn:max_ode1}--\eqref{eqn:max_ode4} in Case (1). First, assume that $c = d$. Then \eqref{eqn:max_ode1} and \eqref{eqn:max_ode3} imply that $c = d > 0$ and that
	\begin{equation}\label{eqn:cequalsd}
		f(u) = \pm \sqrt{c}\,u + \tilde{C}_1\quad\text{and}\quad g(v) = \pm \sqrt{d}\,v + \tilde{C}_2
	\end{equation}
for some real constants of integration $\tilde{C}_1$ and $\tilde{C}_2$.
Now assuming that $c \neq d$, we can explicitly solve for $f(u)$ and $g(v)$ to find that
	\begin{equation}\label{eqn:cneqd}
		\begin{aligned}
		f(u) &= C_1 e^{\sqrt{d - c}\,u} + C_2 e^{-\sqrt{d - c}\,u}, \quad 4(c - d) C_1 C_2 = c,\\
		g(v) &= C_3 e^{\sqrt{c - d}\,v} + C_4 e^{-\sqrt{c - d}\,v}, \quad 4(d - c) C_3 C_4 = d,
		\end{aligned}
	\end{equation}
where $C_1, \ldots, C_4 \in \mathbb{C}$ are constants of integration. Furthermore since $f(u)$ and $g(v)$ are real-valued functions, $C_1, \ldots, C_4$ must satisfy
	\[\begin{cases}
		C_1, C_2 \in \mathbb{R} \text{ and } C_3 = \overline{C_4}, &\text{if }d > c,\\
		C_1 = \overline{C_2} \text{ and } C_3, C_4 \in \mathbb{R}, &\text{if }c > d,
	\end{cases}\]
where $\bar{\cdot}$ denotes the complex conjugation.

To explicitly solve for $f(u)$ and $g(v)$ and hence $\rho(u,v)$, we first need to consider the initial conditions of $f(u)$ and $g(v)$. We identify the exact conditions for $f(u)$ and $g(v)$ having a zero, and derive the appropriate initial conditions in the following series of lemmas.

\begin{lemm}\label{lemm:zero}
$f(u)$ \resp{$g(v)$} satisfying \eqref{eqn:max_ode1}--\eqref{eqn:max_ode4} has a zero if and only if either $c > 0$ or $f(u) \equiv 0$ \resp{$d > 0$ or $g(v) \equiv 0$}.
\end{lemm}

\begin{proof}
If $c = d$, then the statement is trivial by \eqref{eqn:cequalsd}; hence, we may assume $c \neq d$. To prove the necessary condition, since $f \equiv 0$ case is trivial, assume that $c > 0$, and we show that there is some real $u_0$ such that $f(u+0) = 0$. If $d > c$, then it is easy to check that for
	\[
		u_0 := \frac{\log c - \log (4(d-c) C_1^2)}{2 \sqrt{d-c}}
	\]
we get $f(u_0) = 0$ by \eqref{eqn:cneqd}.

Now assume $c > d$. Then since
	\[
		C_1 = \frac{c}{4(c-d) C_2} = \frac{c}{4(c-d) \overline{C_1}},
	\]
we may write $C_1 = \sqrt{\frac{c}{4(c-d}} e^{i\theta}$ and $C_2 = \sqrt{\frac{c}{4(c-d}} e^{-i\theta}$ for some constant $\theta \in \mathbb{R}$. By letting
	\[
		u_0 := \frac{\frac{\pi}{2}-\theta}{\sqrt{c-d}},
	\]
we have $f(u_0) = 0$ again by \eqref{eqn:cneqd}.

To show the sufficient condition, suppose there is some $u_0$ such that $f(u_0) = 0$. By \eqref{eqn:max_ode1}, $(f_u(u_0))^2 = c \geq 0$. If $c = 0$, then, \eqref{eqn:cneqd} gives us
\[f(u) = C_1 e^{\sqrt{d}u} + C_2 e^{-\sqrt{d}u}\]
for some complex constants $C_1$ and $C_2$ where $d\cdot C_1C_2 = 0$. Since $d \neq 0$, without loss of generality, let $C_2 = 0$. From $f(u_0) = 0$, we get $C_1 e^{\sqrt{d}u_0} = 0$. Therefore, $C_1 = 0$, and we have $f(u) \equiv 0$.

The statement regarding $g(v)$ is proven analogously.
\end{proof}

\begin{lemm}
$f(u)$ \resp{$g(v)$} has no zero if and only if either $c < 0$ or $f(u) = \pm e^{\sqrt{d} u}$ where $d > 0$ \resp{$d < 0$ or $g(v) = \pm e^{\sqrt{c} u}$ where $c > 0$}.
\end{lemm}

\begin{proof}
Note that by the previous lemma and the fact that $c < 0$ implies $f(u) \not \equiv 0$, we only need to show that $f(u) \not \equiv 0$ and $c = 0$ if and only if $f(u) = C_1 e^{\sqrt{d} u}$ for $C_1= \pm 1$ and $d >0$.

First, suppose that $f(u) \not \equiv 0$ and $c = 0$. Then, similar to the proof of the previous lemma,
\[f(u) = C_1 e^{\sqrt{d}u}\]
for some complex constant $C_1$. Since $f(u) \not \equiv 0$, $C_1 \neq 0$. In addition, since $f(u)$ is real, $C_1$ is real, and $d > 0$. Finally, by shifting parameters, we may assume that $C_1 = \pm 1$.

Now assume that $f(u) = \pm e^{\sqrt{d} u}$ for $d > 0$. Then $f(u) \not \equiv 0$ trivially. Furthermore, \eqref{eqn:max_ode2} implies that $c\cdot (\pm e^{\sqrt{d} u}) = 0$ for all $u$. Hence, $c = 0$.
\end{proof}

\begin{lemm}
At least one of $f(u)$ or $g(v)$ must have a zero.
\end{lemm}

\begin{proof}
Without loss of generality, suppose that $f(u)$ does not have a zero. Hence, by the previous lemma, $c < 0$ or $f(u) = \pm e^{\sqrt{d} u}$ where $d > 0$. If $f(u) = \pm e^{\sqrt{d} u}$ with $d > 0$, $g(v)$ must have a zero by Lemma \ref{lemm:zero}.

Now suppose $c < 0$. Then, by \eqref{eqn:max_ode1}, $d-c > 0$. If $d < 0$, then \eqref{eqn:max_ode3} implies $c - d > 0$, a contradiction; hence, $d \geq 0$. If $d = 0$, direct calculation shows that either $g(v) \equiv 0$ or $g(v) = C_2 e^{\sqrt{c} u}$ where $C_2 \neq 0$ and $c > 0$. However, since we assumed $c < 0$, it must follow that $g(v) \equiv 0$ or $d > 0$. Hence, $g(v)$ must have a zero.
\end{proof}

Exchanging the roles of $u$ and $v$, if necessary, we may assume without loss of generality that $g$ has a zero, and we may further assume that $g(0) = 0$ by shifting parameters. By considering the fact that we may switch the roles of $f(u)$ and $g(v)$, we only need to consider the following five cases:
	\begin{equation}\label{eqn:fivecases}
		\begin{array}{| c | c | c | c | c |}
			\hline
			c > 0, d > 0 & c > 0, g(v) \equiv 0 & f(u) = \pm e^{\sqrt{d}u}, d > 0 & c < 0, d > 0 & c < 0, g(v) \equiv 0 \\
			\hline
		\end{array}
	\end{equation}
It should be noted that in the cases considered \eqref{eqn:fivecases}, $d \geq 0$, and that $d = 0$ if and only if $g(v) \equiv 0$. For the third case, since $c = 0$,
$g(v) = \sin{(\sqrt{d}v)}.$
By letting $v \mapsto -v$, we see that the plus or minus condition on $f(u)$ may be dropped, allowing us to assume that $f(u) = e^{\sqrt{d}u}$. Finally, we prove the following statement regarding the initial condition of $f(u)$.

\begin{lemm}
For the cases \eqref{eqn:fivecases}, there is some $u_0$ such that $f(u_0) = 1$.
\end{lemm}

\begin{proof}
It is easy to check that the statement holds if $c = d$ via \eqref{eqn:cequalsd}; hence, assume $c \neq d$. From \eqref{eqn:cneqd},
since $c \neq 0$ implies $C_1$ and $C_2$ are non-zero, we let
\[C_2 = \frac{c - d + \sqrt{d(d-c)}}{2(c-d)}.\]

If $d - c > 0$, then since $d \geq 0$, $f(u)$ is real-valued such that $f(0) = 1$. Since $c < 0$ implies that $d - c > 0$, assume $c > 0$ and $d - c < 0$. Then direct calculation shows that $C_1$ is the complex conjugate of $C_2$ implying that $f(u)$ is equal to its own conjugate. Therefore, $f(u)$ is real-valued such that $f(0) = 1$.
\end{proof}

Therefore, through shifting parameters, we may assume that $f(0) = 1$ and $g(0) = 0$. Using these initial conditions, we arrive at the following explicit solutions for $f(u)$ and $g(v)$.

\begin{prop}\label{prop:max_solutionFG}
	For a non-planar maxface $X(u,v)$ with planar curvature lines, the real-analytic solution $\rho: \mathbb{R}^2 \to \mathbb{R}$ of \eqref{eqn:max_pde1} and \eqref{eqn:max_pde2} is precisely given as follows:
	\begin{description}
		\item[Case (1)] If $\Delta \rho \not\equiv 0$, i.e.\ $\Delta \rho$ is not identically equal to zero, then
			\[
				\rho(u,v)= \frac{f(u)^2 + g(v)^2-1}{f_u(u) + g_v(v)},
			\]
		with
			\begin{equation}\label{eqn:explicitform}
				\begin{aligned}
				f(u) &=
					\begin{cases}
            					\cosh{(\sqrt{d - c}\,u)}+ {\dfrac{\sqrt{d}}{\sqrt{d-c}}} \sinh{(\sqrt{d-c}\,u)}, &\text{if }c \neq d\\
						\sqrt{d}u+1, &\text{if } c = d
					\end{cases}\\
            			g(v) &=
					\begin{cases}
						{\dfrac{\sqrt{d}}{\sqrt{d-c}}} \sin{(\sqrt{d - c}\,v)}, &\text{if } c \neq d\\
						\sqrt{d}v, &\text{if } c = d
					\end{cases}
				\end{aligned}
			\end{equation}
		where  $c^2 + d^2 \neq 0$ and $d \geq 0$.

		\item[Case (2)] If $\Delta \rho \equiv 0$, then for some constant $\phi$ such that $\phi \in [0, 2\pi)$,
			\[
				\rho(u,v) = (\cos\phi)\cdot u + (\sin\phi)\cdot v.
			\]
\end{description}
\end{prop}

\begin{proof}
Solving \eqref{eqn:max_ode1}--\eqref{eqn:max_ode4} for $f(u)$ and $g(v)$ with initial conditions $f(0) = 1$ and $g(0) = 0$, and considering the change in parameter $u \mapsto -u$ or $v \mapsto -v$, if necessary, gives the explicit solutions in \eqref{eqn:explicitform}.

Now we wish to see that the domain of $\rho(u,v)$ can be extended to $\mathbb{R}^2$ globally. If $c = d$, then this is a direct result of applying the solution in \eqref{eqn:explicitform}. Therefore, assume $c \neq d$. Then by \eqref{eqn:max_ode1} and \eqref{eqn:max_ode3}, we have
	\[
		f_u^2 - g_v^2 = (d-c)(f^2 + g^2 - 1).
	\]
implying that
	\[
		\rho(u,v) = \frac{f(u)^2 + g(v)^2-1}{f_u(u) + g_v(v)} = \frac{f_u(u) - g_v(v)}{d - c}.
	\]
Therefore, the real-analyticity of $f(u)$ and $g(v)$ implies that the domain of $\rho(u,v)$ can be extended to $\mathbb{R}^2$ globally.
\end{proof}

\begin{remark}\label{rema:maxface}
We make a few important remarks about Proposition \ref{prop:max_solutionFG}:
	\begin{itemize}[label={$\bullet$}]
		\item In the statement of Proposition \ref{prop:max_solutionFG}, we now allow $\rho$ to map into $\mathbb{R}$ as opposed to $\mathbb{R}_+$, i.e.\ $\rho$ may have zeroes, or even be negative.
			By doing so, we now consider $X(u,v)$ as \emph{maxfaces}, defined in \cite{umehara_maximal_2006} as a class of maximal surfaces with singularities (see also \cite{fujimori_singularities_2008}).
		\item In \eqref{eqn:explicitform}, we allow $d - c < 0$. However, even in such case, by using the identities
			\[
				\cosh(\sqrt{d - c}\,u) = \cos(\sqrt{c-d}\,u) \quad\text{and}\quad \sinh(\sqrt{d - c}\,u) = i\sin(\sqrt{c-d}\,u)
			\]
			we see that $f(u)$ and $g(v)$ are real-valued analytic functions.
		\item For case (2) in Proposition \ref{prop:max_solutionFG}, we may use an associated family's parameter $\lambda\in\mathbb{S}^1\subset\mathbb{C}$ instead of $\phi$ through appropriate coordinate change shown below:
			\[\begin{cases}
				\tilde{u}:=\cos\phi\cdot u+\sin\phi\cdot v, & \tilde{v}:=-\sin\phi\cdot u+\cos\phi\cdot v \\
				\lambda:=e^{-i\phi}, &\tilde{Q}:=-\frac{1}{2}\lambda^{-2}=\lambda^{-2}Q.
			\end{cases}\]
			However, it should be noted that while the coordinate change $(u,v)\mapsto(\tilde{u},\tilde{v})$ and parameter change $\phi\mapsto\lambda$ preserve the conformal structure, it does not hold the curvature line coordinates such that $Q\mapsto\tilde{Q}$.
	\end{itemize}
\end{remark}

Note that for all cases, the metric function $\rho(u,v)$ is always bounded for all $(u,v)\in\mathbb{R}^2$, and we now have the following theorem. Note that $u \leftrightarrow v$, used as a subscript in Figure \ref{fig:bifurcation}, means the role of $u$ and $v$ are switched, up to shift of parameters.
\begin{theorem}\label{theo:bifurcation}
Let $X(u,v)$ be a non-planar maxface in $\mathbb{R}^{2,1}$ with isothermic coordinates $(u,v)$ such that the induced metric $\diff s^2 = \rho^2\cdot(\diff u^2 + \diff v^2)$. Then $X$ has planar curvature lines if and only if $\rho(u,v)$ satisfies Proposition \ref{prop:max_solutionFG}. Furthermore, for different values of $(c, d)$ or $\lambda$ as in Remark \ref{rema:maxface}, the surface $X(u,v)$ has the following properties based on Fig. \ref{fig:bifurcation}:
\begin{description}

	\item[Case (1)] If $\Delta \rho \not \equiv 0$, when $(c,d)$ lies on
            \begin{itemize}[label={$\bullet$}]
                \item $\circlenum{1}$: $X$ is not periodic in the $u$-direction, but constant in the $v$-direction,
                \item $\circlenum{2}, \circlenum{3}$, or $\circlenum{4}$: $X$ is not periodic in the $u$-direction, but periodic in the $v$-direction,
                \item $\circlenum{5}$: $X$ is not periodic in both the $u$-direction and the $v$-direction,
                \item $\circlenum{6}$: $X$ is periodic in the $u$-direction, but constant in the $v$-direction.
            \end{itemize}

        	\item[Case (2)]  If $\Delta \rho \equiv 0$, when $\lambda$ lies on
            \begin{itemize}[label={$\bullet$}]
                \item $\circlenum{7}$: $X$ is a surface of revolution,
                \item $\circlenum{8}$: $X$ is a surface in the associated family of $\circlenum{7}$.
            \end{itemize}
\end{description}

\begin{figure}[H]
	\centering
		\scalebox{0.8}{\includegraphics{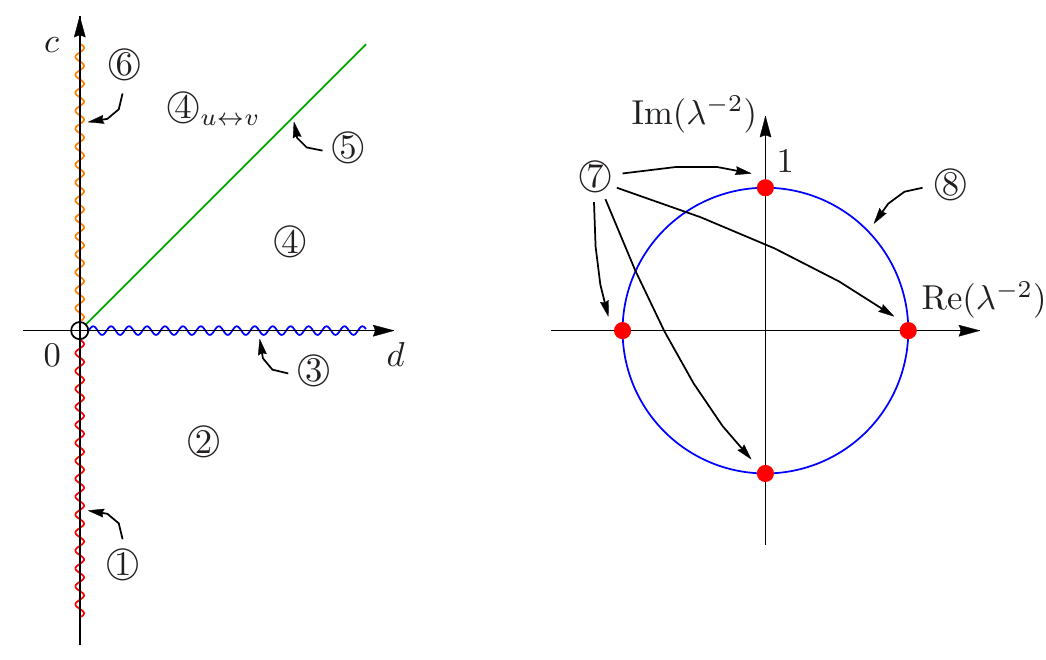}}
     \label{fig:bifurcation}
     \caption{The classification diagrams of non-planar maxfaces with planar curvature lines.}
\end{figure}
\end{theorem}

\subsection{Axial directions and normal vector}\label{sec:walter}
To find the parametrizations of the surfaces considered, we would like to recover the Weierstrass data from the metric function as follows: We first show the existence of a unique constant direction for surfaces under consideration called the \emph{axial direction}, and use it to calculate the unit normal vector. Then from the unit normal vector, we recover the Weierstrass data. First, we show the existence of the axial direction in the following proposition.

\begin{prop}\label{prop:axial}
If there exists $u_0$ \resp{$v_0$} such that $f(u_0)\neq0$ \resp{$g(v_0) \neq 0$} in Proposition \ref{prop:max_solutionFG}, then there exists a unique non-zero constant vector $\vec{v}_1$ \resp{$\vec{v}_2$} such that
	\[
		\langle m(u,v), \vec{v}_1\rangle = \langle m_v(u,v), \vec{v}_1 \rangle = 0
			\quad\text{\resp{$\langle n(u,v), \vec{v}_2\rangle = \langle n_u(u,v), \vec{v}_2 \rangle = 0$}},
	\]
where $m = \rho^{-2} (X_u \times X_{uu})$ (resp. $n = \rho^{-2} (X_v \times X_{vv})$) and
\begin{equation}\label{eqn:v1}
	\vec{v}_1:=\frac{(\rho_u)^2-\rho\cdot\rho_{uu}}{\rho^2}X_u-\frac{\rho_u\rho_v}{\rho^2}X_v+\frac{\rho_u}{\rho}N
\end{equation}
\[\text{\resp{$\vec{v}_2:=-\frac{\rho_u\rho_v}{\rho^2}X_u+\frac{(\rho_v)^2-\rho\cdot\rho_{vv}}{\rho^2}X_v-\frac{\rho_v}{\rho}N$}}.\]
If $\vec{v}_1$ and $\vec{v}_2$ both exist, then they are orthogonal to each other. We call $\vec{v}_1$ and $\vec{v}_2$ the {\em axial directions} of $X(u,v)$.
\end{prop}

\begin{proof}
Similar to the proof of Proposition 2.2 in \cite{cho_deformation_2017}; here we only give an outline of the proof.
Using \eqref{eqn:gaussW}, \eqref{eqn:max_pde1} and \eqref{eqn:max_pde2}, we may calculate that all the required property holds.
%
%
%
\end{proof}

Since $\vec{v}_1$ and $\vec{v}_2$ are constant, we use \eqref{eqn:max_pde1}, \eqref{eqn:max_pde2}, \eqref{eqn:max_solutionRho1}, and \eqref{eqn:max_ode1}--\eqref{eqn:max_ode4} to calculate that
\begin{equation}\label{eqn:v1norm}
	\langle \vec{v}_1, \vec{v}_1\rangle = c, \quad \langle \vec{v}_2 , \vec{v}_2 \rangle = d,
\end{equation}
implying that the axial directions of the surface has the following causalities: if $d > 0$,
\begin{itemize}[label={$\bullet$}]
	\item both $\vec v_1$ and $\vec v_2$ are spacelike if $c > 0$,
	\item $\vec v_1$ is lightlike, but $\vec v_2$ is spacelike if $c = 0$, or
	\item $\vec{v}_1$ is timelike, but $\vec{v}_2$ is spacelike if $c < 0$.
\end{itemize}
Note that if $d = 0$, then $g(v) \equiv 0$, implying that $\vec v_2$ does not exist.
By aligning the axial directions with coordinate axes of the ambient space, we now calculate the unit normal vector.

First, assume $\Delta\rho \not \equiv 0$. Since the definition of $f(u)$ and $g(v)$ depend on the signature of $c - d$, we consider each case separately.

\begin{figure}[H]
    \begin{center}
    \begin{minipage}{0.4\textwidth}
        \centering
        \scalebox{0.9}{\includegraphics{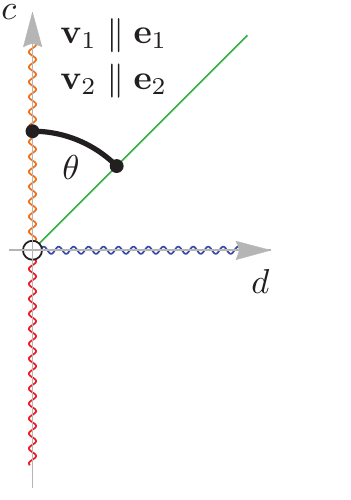}}
    \end{minipage}
    \begin{minipage}{0.4\textwidth}
        \centering
        \scalebox{0.9}{\includegraphics{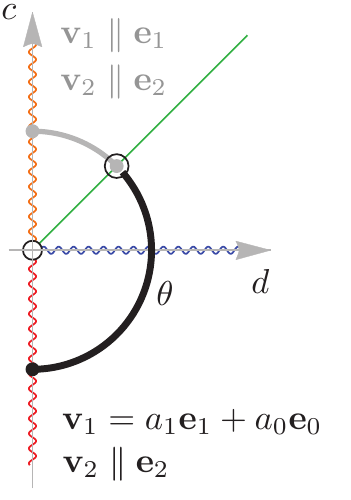}}
    \end{minipage}
    \caption{Choice of parameter and axial directions for cases (1a) and (1b).}
    \label{fig:pathDiag12}
    \end{center}
\end{figure}

\subsubsection{Case (1a)} Assume first that $d - c \leq 0$ (see left side of Fig. \ref{fig:pathDiag12}). Then $\vec{v}_1$ and $\vec{v}_2$ are both spacelike, and we align the axial directions so that $\vec v_1$ and $\vec v_2$ are parallel to $\mathbf e_1$ and $\mathbf e_2$, respectively, where $\mathbf e_i$ are the unit vectors in the $x_i$-direction for $i = 1, 2, 0$. Then, we may calculate the unit normal vector as follows.

\begin{lemm}\label{lemm:normal}
Let $N(u,v) = (N_1, N_2, N_0)$ be the unit normal vector to the surface $X(u,v)$ satisfying case (1) of Proposition \ref{prop:max_solutionFG}. If $c - d \geq 0$, then, the unit normal vector is given by
\[N(u,v) = \left(-\frac{1}{\sqrt{c}}\frac{\rho_u}{\rho},-\frac{1}{\sqrt{d}}\frac{\rho_v}{\rho}, \sqrt{\frac{1}{c}\frac{(\rho_u)^2}{\rho^2} + \frac{1}{d}\frac{(\rho_v)^2}{\rho^2} + 1}\right).\]
\end{lemm}

\begin{proof}
Similar to the proof of Proposition 2.3 in \cite{cho_deformation_2017}; we use the fact that $\langle m(u,v), \vec{v}_1\rangle = \langle m_v(u,v), \vec{v}_1 \rangle = 0$ and $\langle n(u,v), \vec{v}_2\rangle = \langle n_u(u,v), \vec{v}_2 \rangle = 0$ in Proposition \ref{prop:axial}, and the fact that $\langle N, N \rangle = -1$.
%
%
\end{proof}


Now, let $d = r \cos\theta$ and $c = r \sin\theta$ for $\theta \in \left[\frac{\pi}{4},\frac{\pi}{2}\right]$ (see left side of Fig. \ref{fig:pathDiag12}). Since $r$ is a homothety factor of domain $(u,v)$-plane by \eqref{eqn:explicitform}, we may assume $r = 1$. Using the above lemma, we may find the normal vector $N^{\theta}(u,v)$ dependent on $\theta$. On the other hand, since the meromorphic function $h(u,v)$ of the Weierstrass data is equal to the normal vector function under stereographic projection, and since $Q=-\frac{1}{2}(h_u-ih_v)\eta=-\frac{1}{2}$,
\[h(u,v) = \frac{1}{1-N_0(u,v)}(N_1(u,v)+i N_2(u,v)), \quad \eta(u,v) = \frac{1}{h_u - i h_v}.\]
Therefore, using $N^{\theta}(u,v)$, we calculate that for $\theta \in \left[\frac{\pi}{4},\frac{\pi}{2}\right]$,
\begin{equation}\label{eqn:wData1a}
	\begin{aligned}
		h^{\theta}(z) &=
			\begin{cases}
				 \dfrac{A_2 \tan{\left(\frac{1}{2}(A_1 z + A_3)\right)}}{A_1}, &\text{if $\theta \neq \frac{\pi}{4}$}\\
				 2^{-1/4} z + 1, &\text{if $\theta = \frac{\pi}{4}$}
			\end{cases}\\
		\eta^{\theta}(z) &=
			\begin{cases}
				\dfrac{\cos^2{\left(\frac{1}{2}(A_1 z + A_3)\right)}}{A_2}, &\text{if $\theta \neq \frac{\pi}{4}$}\\
				2^{-3/4}, &\text{if $\theta = \frac{\pi}{4}$}
			\end{cases}
	\end{aligned}
\end{equation}
where $A_1(\theta) = \sqrt{\sin\theta - \cos\theta}$, $A_2(\theta) = \sqrt{\cos\theta} + \sqrt{\sin\theta}$, and $A_3(\theta) = \tan^{-1}\left(\sqrt{\tan\theta - 1}\right)$.

\subsubsection{Case (1b)} Now assume that $d - c > 0$ (see right side of Fig. \ref{fig:pathDiag12}). Then, since the causality of $\vec v_1$ changes while that of $\vec v_2$ is always spacelike, we let
	\begin{equation}\label{eqn:a1a0}
		\vec v_1 = a_1\mathbf e_1 + a_0\mathbf e_0
	\end{equation}
for some real constants $a_0$ and $a_1$, while we let $\vec v_2$ be parallel to $\mathbf e_2$. To calculate the unit normal vector $N$ for this case, we first need the following lemma.
\begin{lemm}\label{lemm:normal1b}
	Let $N(u,v) = (N_1, N_2, N_0)$ be the unit normal vector to the surface $X(u,v)$ satisfying case (1) of Proposition \ref{prop:max_solutionFG}.
	If $d - c > 0$, then
		\begin{equation}\label{eqn:n1n0}
			a_1N_1 - a_0N_0 =  -\frac{\sqrt{a_1^2 - a_0^2}}{\sqrt{c}}\,\frac{\rho_u}{\rho}, \quad N_2 = -\frac{1}{\sqrt{d}}\,\frac{\rho_v}{\rho}
		\end{equation}
	where $a_1$ and $a_0$ are as in \eqref{eqn:a1a0}.
\end{lemm}

\begin{proof}
Employing similar techniques to those used in the proof of Proposition 2.3 in \cite{cho_deformation_2017} and Lemma \ref{lemm:normal} implies that
\begin{equation}\label{eqn:n1n0withB}
	a_1N_1 - a_0N_0 = D_2\cdot\frac{\rho_u}{\rho}.
\end{equation}
for some constant $D_2$. To find $D_2$ in \eqref{eqn:n1n0withB}, consider the following system of equations,
\[\begin{cases}
		c = \|((\rho_u)^2-\rho\cdot\rho_{uu})N_u + \rho_u\rho_vN_v+\frac{\rho_u}{\rho}N\|^2\\
		a_1N_1 - a_0N_0 = D_2\cdot\frac{\rho_u}{\rho}\\
		N_1^2 + \frac{1}{d}\,\frac{\rho_v^2}{\rho^2} - N_0^2 = -1
\end{cases}\]
where the first equation comes from \eqref{eqn:gaussW}, \eqref{eqn:v1}, and \eqref{eqn:v1norm}. Since $D_2$ is constant, we may solve for $D_2$ at the point $\left(0, \frac{\pi}{2\sqrt{d-c}}\right)$ to get \eqref{eqn:n1n0}.
\end{proof}

Since $N$ is unit length, \eqref{eqn:n1n0} lets us calculate the normal of the surface for any given $a_0$, $a_1$, $c$, and $d$. Similar to the previous case, let $d = \cos\theta$ and $c = \sin\theta$ for $\theta \in \left[-\frac{\pi}{2},\frac{\pi}{4}\right)$, and further let
\[a_0 = \sqrt{\cos\theta - \sin\theta}, \quad a_1 = \sqrt{\cos\theta}.\]
Now, we calculate the normal vector, and find that for $\theta \in \left[-\frac{\pi}{2},\frac{\pi}{4}\right)$,
\begin{equation}\label{eqn:wData1b}
	h^{\theta}(z) = \dfrac{B_2 e^{B_1 z} - 1}{B_2 - 1}, \quad \eta^{\theta}(z) = \dfrac{(B_2 - 1) e^{-B_1 z}}{2B_1 B_2}
\end{equation}
where $B_1(\theta) = \sqrt{\cos\theta - \sin\theta}$ and $B_2(\theta)  = 1 + \sqrt{1 - \tan\theta}$.

\subsubsection{Case (2)} Finally, we consider the case when $\Delta \rho \equiv 0$. By Remark \ref{rema:maxface}, we only need to consider $\rho(u,v)=u$, and then utilize the parameter $\lambda\in\mathbb{S}^{1}\subset\mathbb{C}$ for the associated family for $\eta(u,v)$. First we assume that the lightlike axis $\vec{v}_1 = \mathbf{e}_1 + \mathbf {e}_0$.
Then similar to the previous cases,
\[N_1 - N_0 = D_3\cdot\frac{\rho_u}{\rho} = D_3\cdot\frac{1}{u}\]
for some real constant $D_3$. By applying the scaling of $u$, without loss of generality, we can assume $D_3 = -1$. Now since $\rho(u,v)=u$ is independent of $v$, we notice that $N(u,v)$ has the form of $N(u,v)=\left(N_1(u),\ 0,\ N_0(u)\right)\cdot T(v)$ for a specific isometry transform $T(v)\in \text{SO}_{2,1}$ keeping the lightlike axis $\mathbf{e}_1 + \mathbf {e}_0$. Thus we get the following:
\begin{lemm}
For case (2) in Proposition \ref{prop:max_solutionFG}, the unit normal vector $N$ of $\tilde{X}$ is given by
\[\begin{aligned}
		N(u,v)&= \left(\frac{u^2 - 1}{2u},\ 0 ,\ \frac{u^2 + 1}{2u}\right)\cdot
			\begin{pmatrix}
				1-\frac{v^2}{2} & v & -\frac{v^2}{2}\\
				-v & 1 & -v\\
				\frac{v^2}{2} & -v & 1+\frac{v^2}{2}
			\end{pmatrix}\\
			&= \left(\frac{u^2 + v^2 - 1}{2u},\ -\frac{v}{u},\ \frac{u^2 + v^2 + 1}{2u}\right).
\end{aligned}\]
\end{lemm}

Using the above proposition, we obtain the following Weierstrass data, up to the homothety of domain:
\begin{equation}\label{eq:wData2}
	\tilde{h}(z) = -\frac{1 + z}{-1 + z}, \quad \tilde{\eta}(z) = \dfrac{1}{4}\,{\left( -1 + z \right) }^2.
\end{equation}
By changing data $(\tilde{h}, \tilde{\eta})\mapsto(\tilde{h},\lambda^{-2}\tilde{\eta})$, we get all maxfaces in case (2) with parameter $\lambda\in\mathbb{S}^{1}$.

Since the data obtained all satisfy the meromorphicity and holomorphicity conditions, we may use the Weierstrass-type representation for maxfaces to obtain the following parametrizations.

\begin{theorem} \label{theo:classification}
If $X(u,v)$ is a non-planar maxface with planar curvature lines in $\mathbb{R}^{2,1}$, then the surface is given by the following parametrization on its domain for some $\theta \in \left[-\frac{\pi}{2},\frac{\pi}{2}\right]$:
\begin{equation}\label{eqn:parametrization1}
X^{\theta}(u,v) =
\begin{cases}
	\begin{pmatrix}
		\frac{(A_1^2 + A_2^2) A_1 u + (A_1^2 - A_2^2)\sin{(A_1 u + A_3)}\cosh{(A_1 v)} + (A_2^2 - A_1^2)\sin A_3}{2 A_1^3 A_2}\\[5pt]
		\frac{(A_2^2 - A_1^2) A_1 v - (A_1^2 + A_2^2)\cos{(A_1 u + A_3)}\sinh{(A_1 v)}}{2 A_1^3 A_2}\\[5pt]
		\frac{\cos{(A_1 u + A_3)}\cosh{(A_1 v)} - \cos{A_3}}{A_1^2}
	\end{pmatrix}^t,&\text{if $\theta \in \left(\frac{\pi}{4},\frac{\pi}{2}\right]$}\\[28pt]

	\begin{pmatrix}
		\frac{e^{-B_1 u}\big\{\big(B_2(B_2(e^{2B_1}-1) + 2) - 2\big)\cos(B_1 v) - 2 e^{B_1 u}\big(B_1 B_2 u + B_2 - 1\big)\big\}}{2(B_1)^2 B_2 (B_2 - 1)}\\[5pt]
		\frac{e^{-B_1 u}\big\{(B_2 e^{2B_1 u} - B_2 + 2)\sin{(B_1 v)} -2 B_1 v e^{B_1 u}\big\}}{2(B_1)^2 (B_2 - 1)}\\[5pt]
		-\frac{B_1 B_2 u + e^{-B_1 u} \cos{(B_1 v)} - 1}{(B_1)^2 B_2} \\
	 \end{pmatrix}^t, &\text{if $\theta \in \left[-\frac{\pi}{2},\frac{\pi}{4}\right)$}\\[28pt]

	\frac{1}{\sqrt{2}}\left(
 		\hat{u} - \hat{u} \hat{v}^2 + \frac{1}{3}\hat{u}^3 - \frac{4}{3},\
		- \hat{v} + \hat{u}^2 \hat{v} - \frac{1}{3} \hat{v}^3,\
		-\hat{u}^2 + \hat{v}^2 + 1
	\right),&\text{if $\theta = \frac{\pi}{4}$}\\
\end{cases}
\end{equation}

%
%
%

where $(\hat{u}, \hat{v})$ is given by $u = 2^{1/4}(\hat{u} - 1)$ and $v = 2^{1/4}\hat{v}$, and
\[\begin{cases}
	A_1(\theta) = \sqrt{\sin\theta - \cos\theta}, A_2(\theta) = \sqrt{\cos\theta} + \sqrt{\sin\theta}, A_3(\theta) = \tan^{-1}\left(\sqrt{\tan\theta - 1}\right)\\
	B_1(\theta) = \sqrt{\cos\theta - \sin\theta}, B_2(\theta)  = 1 + \sqrt{1 - \tan\theta};
\end{cases}\]
or for some $\lambda^{-2} \in \mathbb{S}^1$,
\begin{equation}\label{eqn:parametrization2}
\tilde{X}^\lambda(u,v) =
	\frac{\re (\lambda^{-2})}{2} \begin{pmatrix}
		u - uv^2 + \frac{1}{3}u^3\\
		2uv\\
		-u - uv^2 + \frac{1}{3}u^3
	\end{pmatrix}^t -
	\frac{\im (\lambda^{-2})}{2} \begin{pmatrix}
		v + u^2 v - \frac{1}{3}v^3\\
		- u^2 + v^2\\
		-v + u^2v - \frac{1}{3}v^3
	\end{pmatrix}^t
\end{equation}
up to isometries and homotheties of $\mathbb{R}^{2,1}$.
In fact, it must be a piece of one, and only one, of the following:
\begin{itemize}[label={$\bullet$}]
	\item maximal Enneper-type surface \textup{(E)} with Weierstrass data $(2^{-1/4} z + 1, 2^{-3/4} \diff z)$, $(\theta = \frac{\pi}{4})$,
	\item maximal catenoid with lightlike axis \textup{(C\textsubscript{L})} with Weierstrass data $\left(-\frac{1 + z}{-1 + z}, \frac{1}{4}\,{\left( -1 + z \right) }^2 \diff z \right)$, or a member of its associated family $(\lambda \in \mathbb{S}^1)$,
	\item maximal catenoid with spacelike axis \textup{(C\textsubscript{S})} with Weierstrass data \\$\left(\tan\left(\frac{1}{4}(\pi + 2z)\right), \frac{1}{2}(1 - \sin z) \diff z\right)$, $(\theta = \frac{\pi}{2})$,
	\item maximal catenoid with timelike axis \textup{(C\textsubscript{T})} with Weierstrass data $\left(e^z, \frac{1}{2}e^{-z} \diff z\right)$, $(\theta = -\frac{\pi}{2})$,
	\item maximal Bonnet-type surface with lightlike axial direction \textup{(B\textsubscript{L})} with Weierstrass data \\$\left(2e^z -1, \frac{1}{4}e^{-z} \diff z \right)$, $\left(\theta = 0\right)$,
	\item maximal Bonnet-type surface with spacelike axial direction \textup{(B\textsubscript{S})} with Weierstrass data
	\begin{gather*}
		\left\{\left(\tfrac{A_2 \tan{\left(\tfrac{1}{2}(A_1 z + A_3)\right)}}{A_1}, \tfrac{\cos^2{\left(\tfrac{1}{2}(A_1 z + A_3)\right)}}{A_2} \diff z\right) :  \theta \in \left(\tfrac{\pi}{4},\tfrac{\pi}{2}\right)\right\}, \text{or} \\ 
		\left\{\left(\tfrac{B_2 e^{B_1 z} - 1}{B_2 - 1}, \tfrac{(B_2 - 1) e^{-B_1 z}}{2B_1 B_2} \diff z\right) : \theta \in \left(0,\tfrac{\pi}{4}\right)\right\},
	\end{gather*}
	\item or maximal Bonnet-type surface with timelike axial direction \textup{(B\textsubscript{T})} with Weierstrass data
		\[
			\left\{\left(\tfrac{B_2 e^{B_1 z} - 1}{B_2 - 1}, \tfrac{(B_2 - 1) e^{-B_1 z}}{2B_1 B_2} \diff z\right) : \theta \in \left(-\tfrac{\pi}{2},0\right)\right\}.
		\]
\end{itemize}
Moreover, $X^\theta (u,v)$ is continuous at every point $(u, v)$ with respect to the parameter $\theta$.
\end{theorem}

Note that by \eqref{eqn:v1norm}, we see that the different classes of maximal Bonnet-type surfaces mentioned in \cite{leite_surfaces_2015} have a geometric meaning; namely, the causal character of the axial directions are different. Finally, it should be noted that a catenoid with timelike axis is indeed a limiting case of maximal Bonnet-type surfaces with timelike axial direction, while a catenoid with spacelike axis is a limiting case of maximal Bonnet-type surfaces with spacelike axial direction.

\section{Continuous deformation of maximal surfaces with planar curvature lines} \label{sec:deformation}
Now, we show that all maxfaces with planar curvature lines can be conjoined by a single continuous deformation.
However, as seen in the previous section, the Weierstrass data and the parametrizations of such surfaces depended on two separate parameters $\theta$ and $\lambda$.
Therefore, we need to show that there exists a deformation consisting of maxfaces with planar curvature lines that connects the surfaces in each parameter family. 
In addition, it must be verified that the plane can also be attained as a limit of such surfaces.

Therefore, in this section, we explain how all the maxfaces with planar curvature lines can be joined by a series of continuous deformations.
We consider a deformation to be ``continuous'' with respect to a parameter, if the deformation dependent on the parameter converges uniformally over compact subdomains component-wise.
In fact, it will be enough to show that each component function in the parametrization is continuous for the parameter at any point $(u,v)$ in the domain.

\subsection{Deformation to the maximal catenoid with lightlike axis} First, we will show that there exists a deformation between a maximal Bonnet-type surface with lightlike axial direction and a maximal catenoid with lightlike axis. Assume $c = 0$ and $d > 0$ (see Fig. \ref{fig:pathData3}). Then by \eqref{eqn:v1norm}, $\vec{v}_1$ is a lightlike axial direction while $\vec{v}_2$ is a spacelike axial direction. Therefore, align the vectors as $\vec{v}_1 \parallel \mathbf{e}_1 + \mathbf{e}_0$ and $\vec{v}_2 \parallel \mathbf{e}_2$.

\begin{figure}[H]
	\centering
	\scalebox{1}{\includegraphics{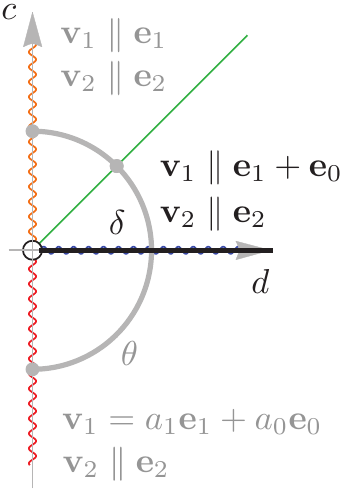}}
	\caption{Choice of parameter and axial directions for deformation to the catenoid with lightlike axis.}
	\label{fig:pathData3}
\end{figure}

Then by Lemma \ref{lemm:normal1b}, $N_1 - N_0 = -\frac{\rho_u}{\rho}$, and $N_2 = -\frac{1}{\delta}\,\frac{\rho_v}{\rho}$ where $\delta = \sqrt d$. Again using the unit normal vector, we calculate the following Weierstrass data:
	\[
		h_{\text{C\textsubscript{L}}}^\delta(z) = \frac{(\delta + 1)e^{\delta z} - 1}{(\delta - 1)e^{\delta z} + 1},
			\quad \eta_{\text{C\textsubscript{L}}}^\delta(z) = \frac{\left(1 + (\delta - 1)e^{\delta z}\right)^2}{4\delta^2 e^{\delta z}}.
	\]
Note that
	\begin{gather*}
		\left.h_{\text{C\textsubscript{L}}}^\delta (z)\right|_{\delta = 1} = \left.h^\theta(z)\right|_{\theta = 0},
			\quad \left.\eta_{\text{C\textsubscript{L}}}^\delta (z)\right|_{\delta = 1} = \left.\eta^\theta(z)\right|_{\theta = 0}, \\
		\lim_{\delta \searrow 0} h_{\text{C\textsubscript{L}}}^\delta (z) = \tilde{h}(z),
			\quad \lim_{\delta \searrow 0} \eta_{\text{C\textsubscript{L}}}^\delta (z) = \tilde{\eta} (z).
	\end{gather*}

In addition, by using the Weierstrass-type representation theorem,
	\[
		X_{\text{C\textsubscript{L}}}^\delta (u,v) =
			\begin{pmatrix}
				\frac{e^{-\delta u}\big\{((\delta^2 + 1)e^{2\delta u} -1) \cos (\delta v) - \delta(2u + \delta)e^{\delta u}\big\}}{2\delta^3}\\[5pt]
				\frac{e^{\delta u} \sin (\delta v) - \delta v}{\delta^2}\\[5pt]
				\frac{e^{-\delta u}\big\{-((\delta^2 - 1)e^{2\delta u} +1) \cos (\delta v) - \delta(2u - \delta)e^{\delta u}\big\}}{2\delta^3}
			\end{pmatrix}^t
	\]
for $\delta > 0$. Since
	\[
		\left.X_{\text{C\textsubscript{L}}}^\delta (u,v)\right|_{\delta = 1} = \left.X^\theta (u,v)\right|_{\theta = 0},
		\quad \lim_{\delta \searrow 0} X_{\text{C\textsubscript{L}}}^\delta (u,v) = \left.\tilde{X}^\lambda(u,v)\right|_{\lambda = 1},
	\]
$X_{\text{C\textsubscript{L}}}^\delta (u,v)$ is a continuous deformation from maximal Bonnet-type surface with lightlike axial direction (B\textsubscript{L}) to the maximal catenoid with lightlike axis (C\textsubscript{L}) (or maximal Enneper-type surface of second kind).

\subsection{Deformation to the plane}
Now, we show that there exists a deformation connecting maximal catenoid with spacelike axis to the plane. Up to this point, to solve the system of ordinary differential equations \eqref{eqn:max_ode1}--\eqref{eqn:max_ode4}, we assumed that $\rho_u(0,v) = f(0) = 1$, and $\rho_v(u,0) = g(0) = 0$. However, since $\rho(u,v) \equiv 1$ for the plane, we must consider different initial conditions for $f(u)$ and $g(v)$. Therefore, we use the result from Lemma \ref{lemm:zero} and consider the surfaces corresponding to case (1a), to assume that $f(0) = 0$ and $g(0) = 0$. Solving \eqref{eqn:max_ode1}--\eqref{eqn:max_ode4} similarly, we get
	\[
		f_\mathrm{P}(u) = \frac{\sqrt{c}}{\sqrt{c - d}} \sin{\left(\sqrt{c - d}\,u\right)}, \quad g_\mathrm{P}(v) = \frac{\sqrt{d}}{\sqrt{c - d}} \sinh{\left(\sqrt{c - d}\,v\right)}
	\]
where $c^2 + d^2 \neq 0$.

Since we assumed each of $f(u)$ and $g(v)$ has a zero, both axial directions are spacelike, and we may use Lemma \ref{lemm:normal} to calculate the unit normal vector. After letting  $\sqrt{c} = \cos\psi$ and $\sqrt{d} = \sin\psi$, we calculate the Weierstrass data as
\begin{align*}
	h_\text{P}^\psi(z) &= \frac{\sqrt{\cos 2\psi}}{\cos\psi - \sin\psi} \tan \left(\frac{\sqrt{\cos 2\psi} }{2}(z + S^\psi)\right)\\
	\eta_\text{P}^\psi(z) &= \frac{1}{\cos\psi + \sin\psi} \cos^2 \left(\frac{\sqrt{\cos 2\psi} }{2}(z+ S^\psi)\right)
\end{align*}
for $\psi \in \left(-\frac{\pi}{4},0 \right]$, where the factor for shifting parameter $S^\psi = 2\psi + \frac{\pi}{2}$ was chosen so that
	\[
		\left.h_\text{P}^\psi(z)\right|_{\psi = 0} = \left.h^\theta(z)\right|_{\theta = \frac{\pi}{2}},
			\quad \left.\eta_\text{P}^0(z)\right|_{\psi = 0} = \left.\eta^\theta(z)\right|_{\theta = \frac{\pi}{2}}.
	\]
Using the Weierstrass-type representation theorem, and multiplying by a homothety factor $\cos 2\psi$, we find that
	\[
		X_\text{P}^\psi (u,v) = \begin{cases}
			\begin{pmatrix}
				\frac{u \cos\psi \sqrt{\cos 2\psi} - \sin\psi \sin\left((u + S^\psi)\sqrt{\cos 2\psi}\right) \cosh \left(v \sqrt{\cos 2\psi}\right)}{\sqrt{\cos 2\psi}}\\[8pt]
				\frac{v \sin\psi \sqrt{\cos 2\psi} - \cos\psi \cos\left((u + S^\psi) \sqrt{\cos 2\psi}\right) \sinh\left(v \sqrt{\cos 2\psi}\right)}{\sqrt{\cos 2\psi}}\\[8pt]
				\cos\left((u + S^\psi) \sqrt{\cos 2\psi}\right) \cosh\left(v \sqrt{\cos 2\psi}\right)
			\end{pmatrix}^t, &\text{if $\psi \neq -\frac{\pi}{4}$}\\[30pt]
			\left(\sqrt{2} u, -\sqrt{2} v, 1\right), &\text{if $\psi = -\frac{\pi}{4}$}
		\end{cases}
	\]
for $\psi \in \left[-\frac{\pi}{4}, 0 \right]$, where $\left. X_\text{P}^{\psi}(u,v) \right|_{\psi = -\frac{\pi}{4}} = \lim_{\psi \searrow -\frac{\pi}{4}} X_\text{P}^\psi (u,v)$. Since
	\[
		\left.X_\text{P}^\psi (u,v)\right|_{\psi = 0} = \left.X^\theta (u,v)\right|_{\theta = \frac{\pi}{2}},
	\]
$X_\text{P}^\psi (u,v)$ defines a continuous deformation from the maximal catenoid with spacelike axis (C\textsubscript{S}) to the plane (P).
In summary, we get the following theorem.

\begin{theorem} \label{theo:oneDeformation}
There exists a continuous deformation consisting precisely of the maxfaces with planar curvature lines (see Fig. \ref{fig:deformationDiag} and Fig. \ref{fig:surfTheta}).
\end{theorem}
\begin{figure}
\begin{center}
	\scalebox{1}{
        \begin{tikzpicture}[scale=1.8]
        \node (A) at (-0.2,1) {P};
        \node (B) at (1,1) {C\textsubscript{S}};
        \node (C) at (1.75,1) {B\textsubscript{S}};
        \node (D) at (2.5,1) {E};
        \node (E) at (3.25,1) {B\textsubscript{S}};
        \node (F) at (4,1) {B\textsubscript{L}};
        \node (G) at (4.75,1) {B\textsubscript{T}};
        \node (H) at (5.5,1) {C\textsubscript{T}};
        \node (I) at (4,0) {C\textsubscript{L}};
        \node (J) at (1,1.2) {$\phantom{\text{C\textsubscript{S}}}$};
        \node (K) at (5.5,1.2) {$\phantom{\text{C\textsubscript{T}}}$};
        \path[<->]
        (A) edge node[above]{$X_\text{P}^\psi$}(B)
        (B) edge (C)
        (C) edge (D)
        (D) edge (E)
        (E) edge (F)
        (F) edge (G)
        (G) edge (H)
        (J) edge node[above]{$X^\theta$}(K)
        (F) edge node[right]{$X_{\text{C\textsubscript{L}}}^\delta$} (I);
        \draw [<->] (3.73,-0.1) arc (-230:50:0.4);
        \node (L) at (4,-0.6) {$\tilde{X}^\lambda$};
        \end{tikzpicture}}
\end{center}
\caption{Diagram of deformations connecting maxfaces with planar curvature lines.}
\label{fig:deformationDiag}
\end{figure}
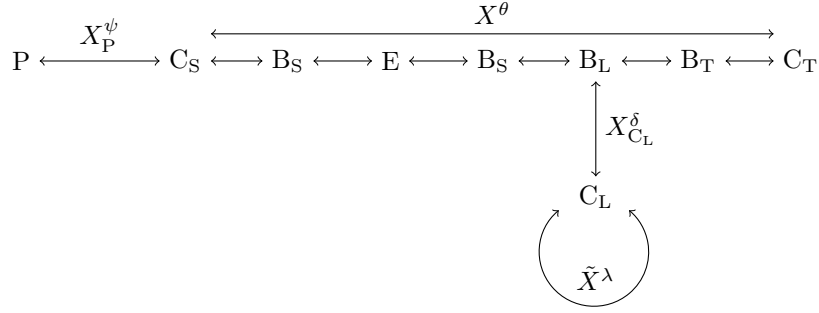

\begin{figure}
	\centering
	\scalebox{0.88}{\includegraphics{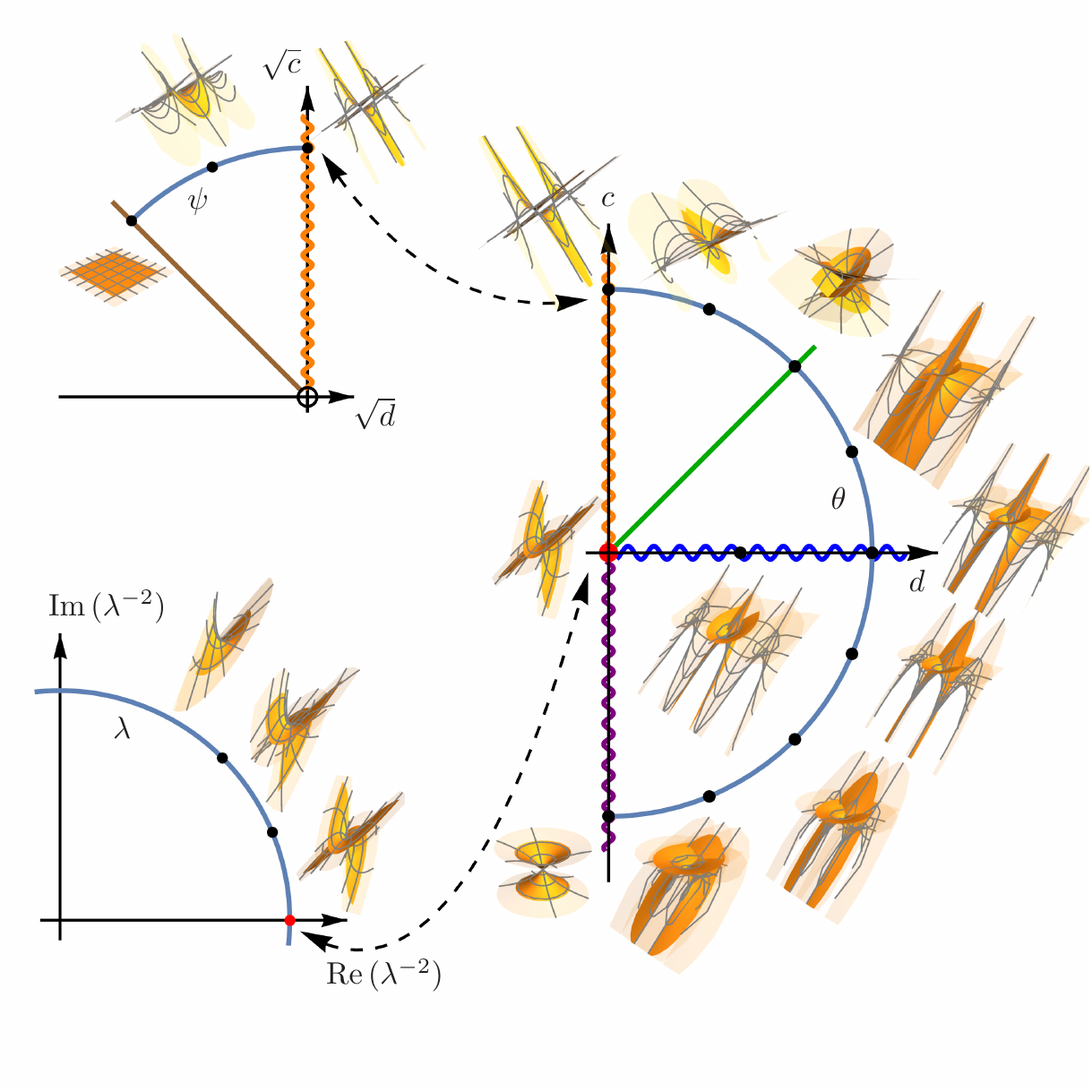}}
	\caption{Maxfaces with planar curvature lines and their deformations.}
	\label{fig:surfTheta}
\end{figure}


\section{Singularities of maximal Bonnet-type surfaces} \label{sec:singularities}

As mentioned in Remark \ref{rema:maxface}, \emph{maxfaces} was introduced as a class of maximal surfaces with singularities in \cite{umehara_maximal_2006}. In this section, we investigate the types of singularities appearing on maxfaces with planar curvature lines. Since the types of singularities of maximal catenoids and maximal Enneper-type surfaces have been investigated \cite{kobayashi_maximal_1984, umehara_maximal_2006, fujimori_singularities_2008}, we focus on recognizing the types of singularities on maximal Bonnet-type surfaces.

Let $S(X):=\{(u,v)\in\mathbb{R}^2 : \rho(u,v)=0\}$ be the singular set. Then using the explicit solution of the metric function in Proposition \ref{prop:max_solutionFG}, we understand that the singular set becomes $1$-dimensional. To recognize the types of singularities of maximal Bonnet-type surfaces, we refer to the following results from \cite{umehara_maximal_2006}, \cite{fujimori_singularities_2008}, and \cite{ogata_duality_}.

\begin{fact}\label{prop:criteria}
Let $X(u,v) : \Sigma \to \mathbb{R}^{2,1}$ be a maxface with Weierstrass data $(h, \eta\,\diff z)$. Then, a point $p \in \Sigma$ is a singular point if and only if $| h(p) | = 1$. Furthermore, for
\[\varphi := \frac{h_z}{h^2 \eta}, \quad \phi := \frac{h}{h_z}\varphi_z, \quad \Phi := \frac{h}{h_z}\phi_z,\]
the image of $X$ around a singular point $p$ is locally diffeomorphic to
\begin{itemize}
	\item[$\bullet$] a cuspidal edge if and only if $\re \varphi \neq 0$ and $\im \varphi \neq 0$ at $p$,
	\item[$\bullet$] a swallowtail if and only if $\varphi \in \mathbb{R} \setminus \{0\}$ and $\re \phi \neq 0$ at $p$,
	\item[$\bullet$] a cuspidal cross cap if and only if $\varphi \in i\,\mathbb{R} \setminus \{0\}$ and $\im \phi \neq 0$ at $p$, or
	\item[$\bullet$] a cuspidal $S_1^-$ singularity if and only if $\varphi \in i\,\mathbb{R} \setminus \{0\}$, $\phi \in \mathbb{R} \setminus \{0\}$, and $\im \Phi \neq 0$ at $p$.
\end{itemize}
\end{fact}

To make the calculations simpler, from Lemma \ref{lemm:normal1b}, assume that $c = t^2 - 1$ and $d = t^2$ for $t > 0$. If we further assume that $a_1 = t$ and $a_0 = 1$,
then we obtain the following Weierstrass data:
\begin{equation}\label{eqn:sing_Wdata}
	h^t(z) = e^z - t, \quad \eta^t(z) = \frac{e^{-z}}{2}
\end{equation}
after a shift of paramter $u \mapsto u + \log(t+1)$. Note that this Weierstrass data represents exactly the Bonnet family described in \cite{leite_surfaces_2015}, and that all maximal Bonnet-type surfaces are included in this family by Theorem \ref{theo:bifurcation}. Then, for the family,
\begin{itemize}
	\item[$\bullet$] If $t > 1$, the surface is a maximal Bonnet-type surface with spacelike axial direction.
	\item[$\bullet$] If $t = 1$, the surface is a maximal Bonnet-type surface with lightlike axial direction.
	\item[$\bullet$] If $t < 1$, the surface is a maximal Bonnet-type surface with timelike axial direction.
\end{itemize}

By directly calculating $\varphi$, $\phi$, and $\Phi$, and using Fact \ref{prop:criteria}, we arrive at the following result.
\begin{prop}
Let $X^t(u,v)$ be a maximal Bonnet-type surface with the Weierstrass data given in \eqref{eqn:sing_Wdata}. The image of $X$ around a singular point $p = (u,v)$ is locally diffeomorphic to the following:
\begin{itemize}
\item[$\bullet$] a swallowtail (SW) only at
\[\renewcommand\arraystretch{1.1}
\begin{array}{ c  c  l }
	0 < t < 1 & \textup{(B\textsubscript{T}):} & \left(\log(\circlepm\, t + 1), \cos^{-1} (\circlepm\, 1)\right)\\
	t = 1 & \textup{(B\textsubscript{L}):} & (\log 2, \cos^{-1} 1)\\
	t > 1 &\textup{(B\textsubscript{S}):} &  (\log (t \pm 1), \cos^{-1} 1), \left(\log \sqrt{t^2 - 1}, \cos^{-1}\sqrt{1-t^{-2}}\right)
\end{array}\]
\item[$\bullet$] a cuspidal cross cap (CCR) only at
\[\renewcommand\arraystretch{1.5}
\begin{array}{ c  c  l }
	0 < t \leq \frac{1}{\sqrt{2}} & \textup{(B\textsubscript{T}):} & \text{None}\\
	\frac{1}{\sqrt{2}} < t < 1 & \textup{(B\textsubscript{T}):} & \left(\log\left(\circlepm\sqrt{t^2 - \frac{1}{2}} + \sqrt{\frac{1}{2}}\right), \cos^{-1}\left(\circlepm\frac{1}{t}\sqrt{t^2 - \frac{1}{2}}\right)\right)\\
	t = 1 & \textup{(B\textsubscript{L}):} & \left(\log \sqrt{2}, \cos^{-1} \frac{1}{\sqrt{2}}\right)\\
	t > 1 & \textup{(B\textsubscript{S}):} & \left(\log\left(\sqrt{t^2 - \frac{1}{2}} \pm \sqrt{\frac{1}{2}}\right),  \cos^{-1}\left(\frac{1}{t}\sqrt{t^2 - \frac{1}{2}}\right)\right)
\end{array}\]
\item[$\bullet$] or a cuspidal $S_1^-$ singularity (CS) only at
\[\renewcommand\arraystretch{1}
\begin{array}{ c  c l }
	t = 1/\sqrt{2} & \textup{(B\textsubscript{T}):} & \left(-\log(\sqrt{2}), \cos^{-1} 0\right)
\end{array}\]
\end{itemize}
where $\circlepm$ corresponds to each other.
\end{prop}
Hence, from the singularity theory point of view, we understand that maximal Bonnet-type surfaces with timelike axial directions can further be classified into the following three types: type 1 (B\textsubscript{T1}), type 2 (B\textsubscript{T2}), or type 3 (B\textsubscript{T3}).

Since maximal Bonnet-type surfaces are periodic in the $v$-direction, let a single portion of $X(u,v)$ refer to the part of the surface mapped over a single period of $v$ in the domain. Then, in summary, we understand the following theorem concerning the types of singularities on maximal Bonnet-type surfaces.

\begin{theorem}\label{theo:singularityType}
Let $X^t(u,v)$ be a maximal Bonnet-type surface with the Weierstrass data given in \eqref{eqn:sing_Wdata}. The images of a single portion of $X$ around singular points are locally diffeomorphic to cuspidal edges except at the following number of points.
\[\renewcommand\arraystretch{1}
\begin{array}{@{} ccccc @{}}
	\toprule
		& \text{type of surface} & \text{\# of SW} & \text{\# of CCR} & \text{\# of CS}\\
	\midrule
	0 < t < 1/\sqrt{2} & \textup{B\textsubscript{T1}} & 2 & 0 & 0\\
	t = 1/\sqrt{2} & \textup{B\textsubscript{T2}} & 2 & 0 & 2\\
	1/\sqrt{2} < t < 1 & \textup{B\textsubscript{T3}} & 2 & 4 & 0\\
	t = 1 & \textup{B\textsubscript{L}} & 1 & 2 & 0\\
	1 < t & \textup{B\textsubscript{S}} & 4 & 4 & 0\\
	\bottomrule
\end{array}\]
\end{theorem}

Combined with the result in \cite{kobayashi_maximal_1984}, \cite{umehara_maximal_2006}, and \cite{fujimori_singularities_2008}, we obtain the following corollary.
\begin{coro}
Let $X(u,v)$ be a maxface with planar curvature lines. If $p$ is a singular point of $X(u,v)$, then the image of $X$ around the singular point $p$ must be locally diffeomorphic to one of the following: conelike singularity, fold singularity, cuspidal edge, swallowtail, cuspidal cross cap, or cuspidal $S_1^-$ singularity.
\end{coro}

\begin{figure}
    \centering
    \scalebox{0.9}{\includegraphics{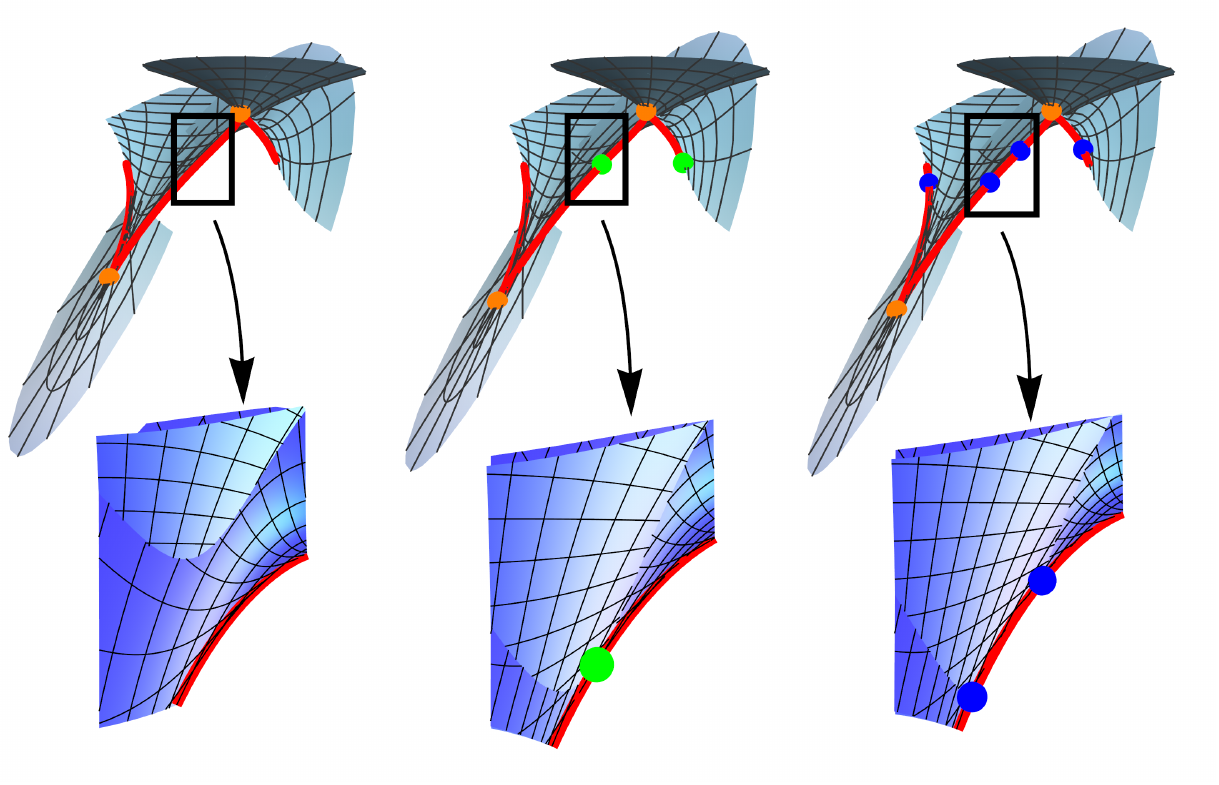}}
    \caption{Types of singularities for maximal Bonnet-type surfaces with timelike axial directions (\text{B\textsubscript{T1}}, \text{B\textsubscript{T2}}, \text{B\textsubscript{T3}}) where the cuspidal edges are highlighted by a red line, swallowtails by orange points, cuspidal cross caps by blue points, and cuspidal $S_1^-$ singularities by green points.}
    \label{fig:singDiag1}
\end{figure}

\begin{figure}
    \centering
    \scalebox{0.9}{\includegraphics{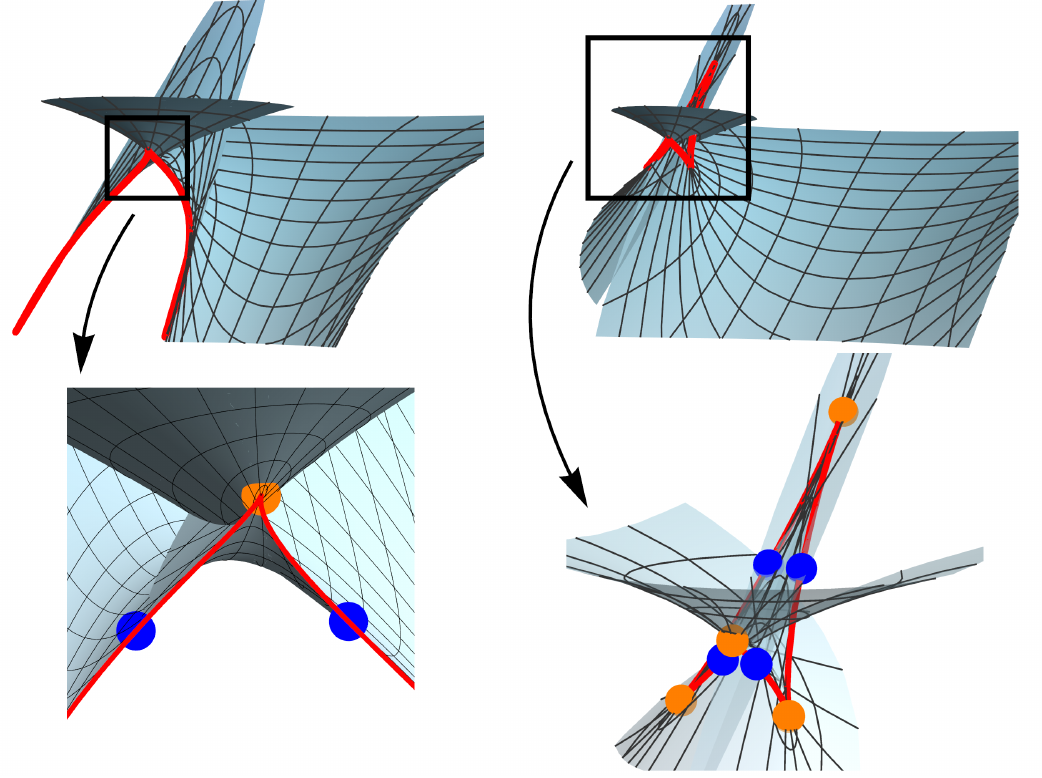}}
    \caption{Types of singularities for maximal Bonnet-type surfaces with lightlike and space axial directions (\text{B\textsubscript{L}}, \text{B\textsubscript{S}}) where the cuspidal edges are highlighted by a red line, swallowtails by orange points, and cuspidal cross caps by blue points.}
    \label{fig:singDiag2}
\end{figure}

\section{Maximal surfaces that are also affine minimal surfaces}\label{sec:thomsen}
In the Euclidean case, Thomsen studied minimal surfaces in $\mathbb{R}^3$ that are also affine minimal surfaces, those surfaces with zero affine mean curvature surfaces and with indefinite affine metric with respect to the equiaffine structure, called Thomsen surfaces, in \cite{thomsen_uber_1923}, and commented on the fact that such surfaces are conjugates of minimal surfaces with planar curvature lines. The analogous statement holds true for maximal surfaces in $\mathbb{R}^{2,1}$ as shown through the following result in \cite{manhart_bonnet-thomsen_2015}.

\begin{fact}
An umbilic-free maximal surface in $\mathbb{R}^{2,1}$ has planar curvature lines if and only if the conjugate surface is an affine minimal surface.
\end{fact}

Therefore, by considering the conjugate surfaces of maximal surfaces with planar curvature lines, we get the following result from Theorem \ref{theo:oneDeformation}.

\begin{coro2}[Corollary to Theorem \ref{theo:oneDeformation}]\label{coro:thomsen1}
There exists a continuous deformation consisting precisely of the maximal surfaces that are also affine minimal surfaces.
\end{coro2}

Furthermore, by the duality of singularities between conjugate surfaces explored in \cite{umehara_maximal_2006}, \cite{kim_prescribing_2007}, \cite{fujimori_singularities_2008}, and \cite{ogata_duality_}, we obtain the following classification of singularities on maximal Thomsen-type surfaces.

\begin{coro2}[Corollary to Theorem \ref{theo:singularityType}]\label{coro:thomsen2}
Let $Y^t(u,v)$ be a maximal Thomsen-type surface where $Y^t(u,v)$ is the conjugate surface of $X^t(u,v)$ as defined in Theorem \ref{theo:singularityType}. The images of a single portion of $Y$ around singular points are locally diffeomorphic to cuspidal edges except at the following number of points.

	\[\begin{array}{@{} cccc @{}}
		\toprule
		 & \text{\# of CCR} & \text{\# of SW} & \text{\# of CB}\\
		 \midrule
		 0 < t < 1/\sqrt{2} & 2 & 0 & 0\\
		 t = 1/\sqrt{2} &2 & 0 & 2\\
		 1/\sqrt{2} < t < 1 &  2 & 4 & 0\\
		 t = 1 &  1 & 2 & 0\\
		 1 < t & 4 & 4 & 0\\
		 \bottomrule
	\end{array}\]
where CB stands for cuspidal butterfly.

Moreover, if $Y(u,v)$ is a maximal surface that is also an affine minimal surface, then the image of $Y$ around the singular point $p$ must be locally diffeomorphic to one of the following: conelike singularity, fold singularity, cuspidal edge, swallowtail, cuspidal cross cap, or cuspidal butterfly.
\end{coro2}

\begin{figure}
    \centering
    \scalebox{0.9}{\includegraphics{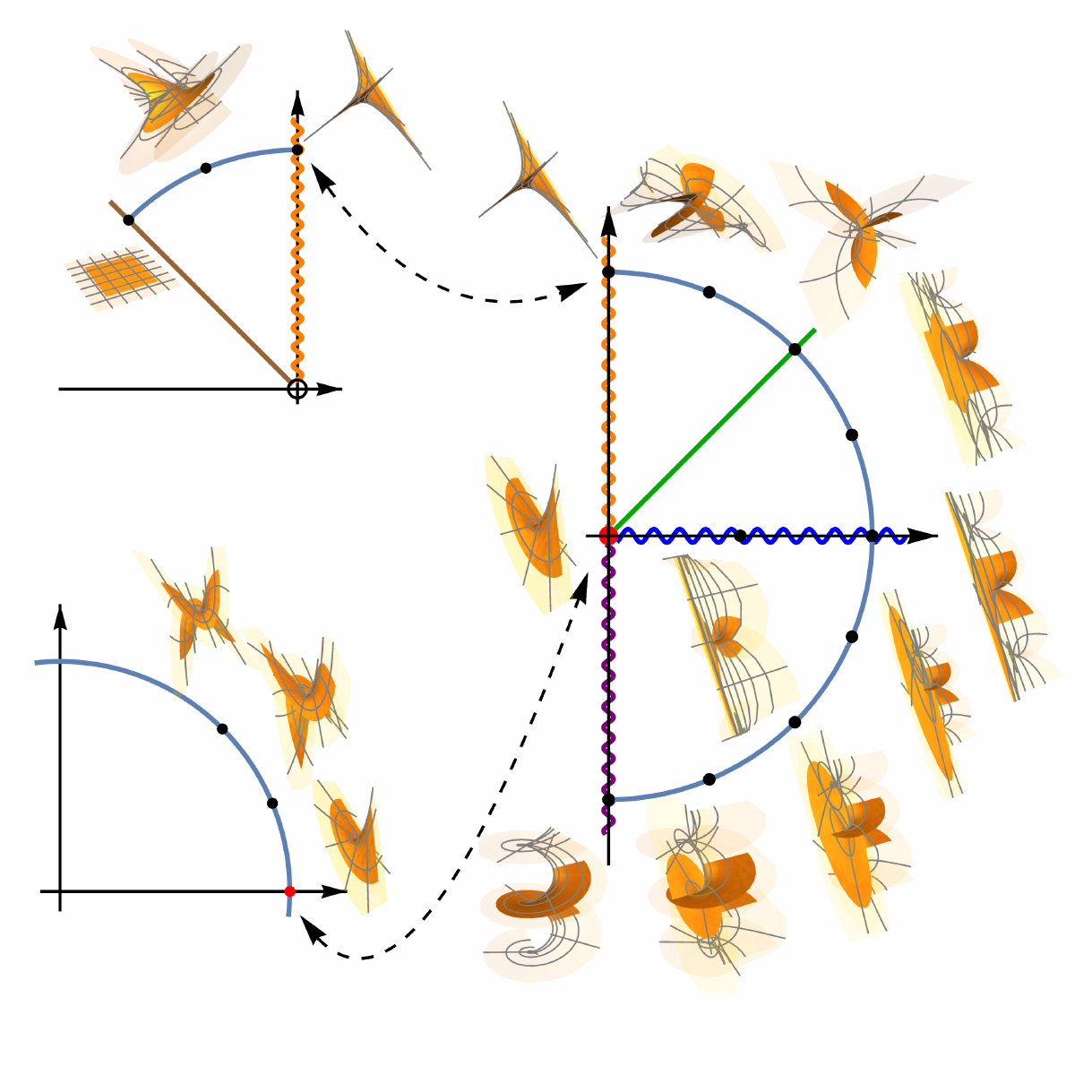}}
    \caption{Maximal surfaces that are also affine minimal and their deformations.}
    \label{fig:surfConj}
\end{figure}

\begin{figure}
    \begin{center}
    \begin{minipage}{0.5\textwidth}
        \centering
        \scalebox{0.5}{\includegraphics{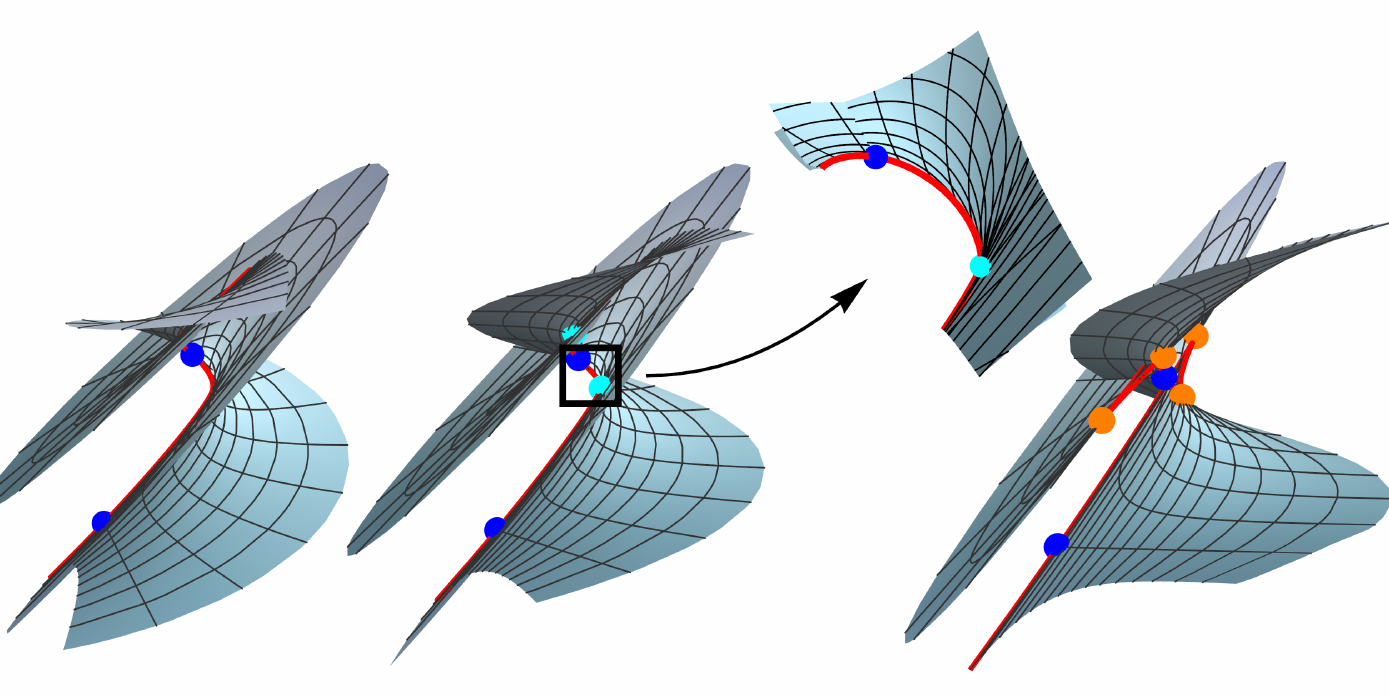}}
    \end{minipage}
    \begin{minipage}{0.4\textwidth}
        \centering
        \scalebox{0.5}{\includegraphics{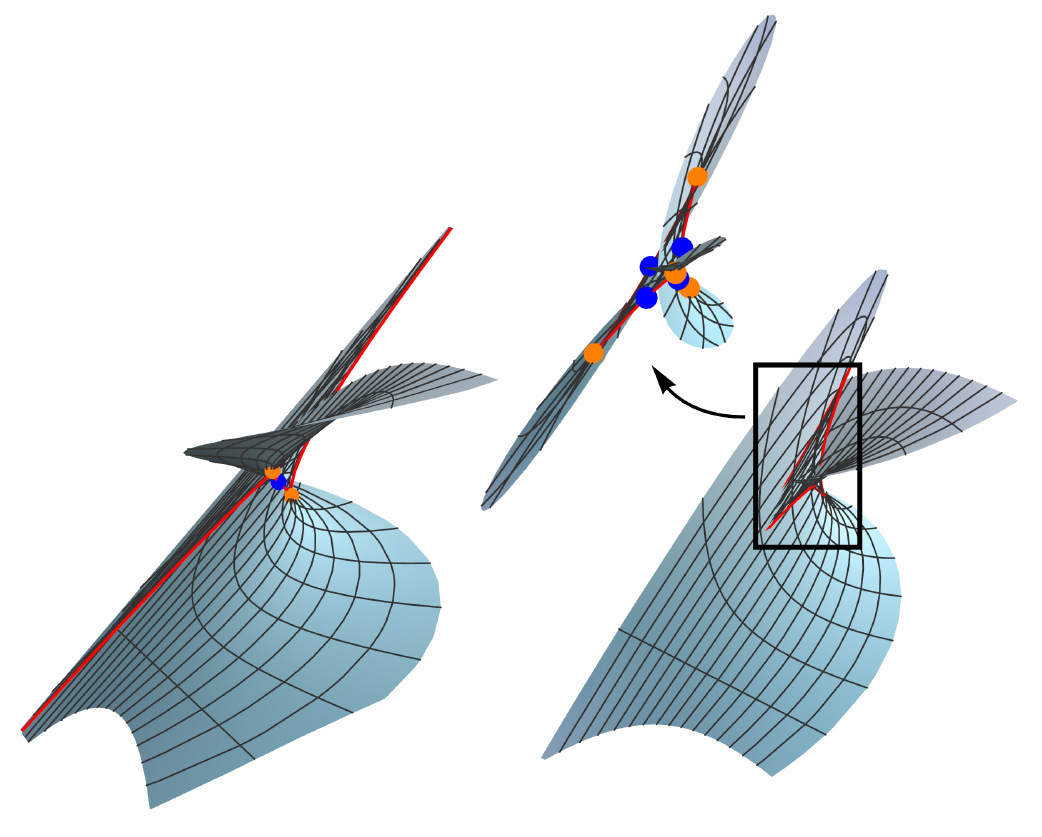}}
    \end{minipage}
    \caption{Types of singularities for maximal Thomsen-type surfaces where the cuspidal edges are highlighted by a red line, swallowtails by orange points, cuspidal cross caps by blue points, and cuspidal butterflies by cyan points.}
    \label{fig:singDiagConj1}
    \end{center}
\end{figure}

%

\vspace{15pt}
{\bf Acknowledgements.} The authors express their gratitude to Professor Shoichi Fujimori, Professor Hitoshi Furuhata, Professor Wayne Rossman, Keisuke Teramoto, and the referee for many useful comments.

\nocite{*}
\bibliographystyle{abbrvnat}
\bibliography{joseph-ogata2.bib}

\end{document}